\documentclass[10pt, twosides, a4paper]{amsart}

\usepackage{graphicx}
\usepackage{amsmath,amssymb,amsfonts,mathtools}
\usepackage{subcaption}
\usepackage{tikz}
\usetikzlibrary{calc}
\usepackage{color}

\usepackage{url}

\usepackage{natbib}

\usepackage[algo2e,boxruled]{algorithm2e}

\usepackage[font=footnotesize]{caption}

        {\hspace*{\fill}$\Box$\par}

\usepackage{enumitem}
\newlist{ass}{enumerate}{1}
\setlist[ass]{label = {\bf (A\arabic*)}, resume}
\newlist{assBV}{enumerate}{1}
\setlist[assBV]{label = {\bf (A\arabic*-BV)}, resume}
\newlist{hyp}{enumerate}{1}
\setlist[hyp]{label = {\bf (H\arabic*)}, resume}
\newlist{obj}{enumerate}{1}
\setlist[obj]{label = {\bf (O\arabic*)}, resume}
\newlist{lyap}{enumerate}{1}
\setlist[lyap]{label = {\bf (L\arabic*)}, resume}

\theoremstyle{plain}
\newtheorem{lem}{Lemma}
\newtheorem{prop}{Proposition}
\newtheorem{thm}{Theorem}

\theoremstyle{definition}
\newtheorem{defn}{Definition}
\newtheorem{exam}{Example}

\theoremstyle{remark}

\newcommand{\N}{\mathbb{N}}

\newcommand{\R}{\mathbb{R}}

\newcommand{\cC}{\mathcal{C}}

\newcommand{\cI}{\mathcal{I}}

\newcommand{\cK}{\mathcal{K}}

\newcommand{\cM}{\mathcal{M}}
\newcommand{\cN}{\mathcal{N}}
\newcommand{\cO}{\mathcal{O}}
\newcommand{\cP}{\mathcal{P}}

\newcommand{\cS}{\mathcal{S}}
\newcommand{\cT}{\mathcal{T}}

\DeclareMathOperator{\bd}{bd}

\DeclareMathOperator{\inn}{int}

\newcommand{\eps}{\varepsilon}

\newcommand{\ul}[1]{\underline #1}
\newcommand{\ol}[1]{\overline #1}

\colorlet{Darkred}{red!50!black}
\colorlet{Darkgreen}{green!50!black}
\xdefinecolor{UIUCblue}{rgb}{0,0.24,0.49} 
\xdefinecolor{UIUCorange}{rgb}{0.96,0.5,0.14} 

\author[M. Souaiby]{Marianne Souaiby}
\address{LAAS -- CNRS, University of Toulouse\\ 31400 Toulouse\\ France}
\email{msouaiby@laas.fr}

\author[A. Tanwani]{Aneel Tanwani}
\address{LAAS -- CNRS, University of Toulouse\\ 31400 Toulouse\\ France}
\email{aneel.tanwani@laas.fr}
\urladdr{http://homepages.laas.fr/atanwani}

\author[D. Henrion]{Didier Henrion}
\address{LAAS -- CNRS, University of Toulouse\\ 31400 Toulouse\\ France; and Fac. Elec. Engr., Czech Tech. Univ. Prague, Czechia.}
\email{henrion@laas.fr}
\urladdr{http://homepages.laas.fr/henrion}

\title[Semidefinite Programs for Lyapunov Functions]{Computation of Lyapunov Functions under State Constraints using Semidefinite Programming Hierarchies}

\thanks{This work is supported by ANR project {\sc ConVan}, grant number ANR-17-CE40-0019-01.}

\begin{document}

\maketitle

\begin{abstract}
We provide algorithms for computing a Lyapunov function for a class of systems where the state trajectories are constrained to evolve within a closed convex set. The dynamical systems that we consider comprise a differential equation which ensures continuous evolution within the domain, and a normal cone inclusion which ensures that the state trajectory remains within a prespecified set at all times. Finding a Lyapunov function for such a system boils down to finding a function which satisfies certain inequalities on the admissible set of state constraints. It is well-known that this problem, despite being convex, is computationally difficult. For conic constraints, we provide a discretization algorithm based on simplicial partitioning of a simplex, so that the search of desired function is addressed by constructing a hierarchy (associated with the diameter of the cells in the partition) of linear programs. Our second algorithm is tailored to semi-algebraic sets, where a hierarchy of semidefinite programs is constructed to compute Lyapunov functions as a sum-of-squares polynomial.
\end{abstract}


\section{Introduction}

Constrained dynamical systems, where the evolution of state trajectories is confined to a predefined set, arise in different applications. 
Mathematically, given a closed convex set $\cS \subset \R^n$, and a continuously differentiable function $f:\R^n \to \R^n$, one possible way to describe the evolution of constrained systems is via the differential inclusion
\begin{equation}\label{eq:compSysGen}
\dot x \in f(x) - \cN_\cS (x)
\end{equation}
where $\cN_\cS (x) \in \R^n$ denotes the outward normal cone to the set $\cS$ at the point $x \in \R^n$. 
An absolutely continuous function $x:[0,T] \to \R^n$ is a solution of \eqref{eq:compSysGen} if there exists a (possibly discontinuous and state-dependent) function $\eta:[0,T] \to \R^n$ such that \eqref{eq:compSysGen} holds for almost every $t$, and $\eta(t) \in -\cN_\cS(x(t))$, for all $t \ge 0$.
In other words, if at a time $t \in [0,T]$, $x(t)$ is in the interior of $\cS$, then $\eta(t)$ is essentially equal to $0$. However, if $x(t)$ is on the boundary of set $\cS$, then the vector $\eta(t) \in -\cN_\cS(x(t))$ is chosen such that $\dot x(t) = f(x(t)) +\eta(t)$ points inside the set $\cS$, which allows the motion to continue within the set $\cS$.  In other words, one can also interpret the evolution of the trajectories of system~\eqref{eq:compSysGen} to be constrained in such a manner that $x(t) \in \cS$, for each $t \ge 0$.

System of form~\eqref{eq:compSysGen} are naturally related to {\em projected dynamical systems}, and connections can be drawn between constrained system \eqref{eq:compSysGen} and other classes of nonsmooth systems \citep{BrogTanw20}. By and large, such models have found useful applications in the modeling of electrical circuits, mechanical systems with impacts \citep{acary2011,adly2017,leine2008}. Some variants of these systems are also studied in \citep{TanwBrog18} in the context of estimation and output regulation problems.

In this article, we are interested in developing tools for stability analysis of constrained system using Lyapunov functions. In the literature, we basically find two approaches for finding Lyapunov functions for systems of form \eqref{eq:compSysGen} which mostly focus on $f(x) = Ax$, with $A \in \R^{n \times n}$ being a matrix. In the first approach, the nonsmooth multiplier $\eta$ is seen as a nonlinearity in the feedback and the interaction of the differential equation with the nonlinearity $\eta$ is interpreted as a Lur'e system. With this perspective, and under some {\em passivity} assumptions, the resulting Lur'e system is shown to be asymptotically stable with a quadratic positive definite Lyapunov function. However, this approach imposes certain structural requirements on the system dynamics, which is not always desirable. If one works without such assumptions, then it is important to take into account the constraints imposed on the state trajectories of system \eqref{eq:compSysGen}. With the constraint set being the positive orthant, this second viewpoint is directed at computing the copositive Lyapunov functions (which are positive definite only in the positive orthant, and not necessarily in the entire state space). 
Sufficient conditions for finding copositive Lyapunov functions of complementarity systems appear in \citep{GoelMotr03, GoelBrog04, CamlPang06}, where once again these conditions are aimed at finding copositive matrices with $f(\cdot)$ being linear, so that the quadratic form associated with the resulting copositive matrices describes a Lyapunov function. There is no obvious indication in these works about how they could be generalized to nonlinear $f(\cdot)$ in \eqref{eq:compSysGen}. Even if such conditions could be formulated, the computational complexity of checking such conditions remains unknown.

Stability analysis using copositive Lyapunov functions, for system~\eqref{eq:compSysGen} with $f(\cdot)$ nonlinear and $\cS$ being the positive orthant of $\R^n$, have been addressed in our recent work \citep{SouaTanw19pp} where we have shown that, for a certain class of complementarity systems, if the origin is globally exponentially stable then there exists a continuously differentiable copositive Lyapunov function. For computing such a function numerically. In this paper, we extend those ideas to study more general constraint sets $\cS$ and how our earlier algorithms can be adapted for these broader class of sets.

The first of these algorithms corresponds to checking the inequalities, associated with the search of Lyapunov function, over a finite set only. Here, we present a modification of this algorithm when the constraint set is a polyhedral cone, instead of just the positive orthant. In the literature, similar ideas have been used for checking copositivity of a quadratic form \citep{BunDur08, BunDur09, NiYaZha18}. Among these, the algorithms proposed in \citep{BunDur08, BunDur09} are based on using the homogenous structure of the function, and via suitable partitioning, transform the problem of checking copositivity to solving a linear set of equations. We adopt this philosophy in our work here as well.  In particular, if the vector field $f$ in \eqref{eq:compSysGen} is homogenous, our search boils down to computing a function which satisfies a given set of inequalities on the {\em standard simplex} only. We partition this simplex in an appropriate way, and check the inequalities at the discrete nodes of the partition, and solve for the coefficients of the desired polynomial copositive Lyapunov function. As we increase the nodes in our partition, a converging hierarchy of linear programs is established. The results are backed up by academic examples.

The other method that we use is based on {\em sum-of-squares} (SOS) decomposition of the Lyapunov function. While checking if a function is positive everywhere is numerically hard, checking if it admits an SOS decomposition is a semidefinite program \citep{Powers98}. Hence, numerical tools based on SOS optimization have been developed extensively over the past two decades to compute Lyapunov functions, see e.g. \citep{Parr00, PrajPapa02, HenrGaru05, Chesi09}. In the context of systems with switching vector fields, the construction of Lyapunov functions using SOS is studied in \cite{PapaPraj09, AhmPar17, AhmaJung18}. An overview of sum-of-squares techniques can be found in \citep{Lass15}, and applications of semidefinite programming for solving polynomial inequalities in control systems related problems appear in \citep{Henr13}.
In contrast to our work in \citep{SouaTanw19pp} with constraints being positive orthants only, we consider compact semi-algebraic sets in this work. We provide sufficient conditions which alleviate the need to compute the analytic solution to a quadratic optimization problem (on the boundary of $\cS$) to check the corresponding Lyapunov inequalities. This allows us to overcome the computational burden associated with the SOS technique proposed in \citep{SouaTanw19pp}.


\section{Problem Setup}\label{sec:prob}

In this article, we focus on a particular class of constrained systems described by
\begin{subequations}\label{eq:compSys}
\begin{gather}
\dot x = f(x) + \eta \label{eq:compSysa}\\
\eta \in -\cN_\cS(x), \label{eq:compSysb}
\end{gather}
\end{subequations}
where $f:\R^n \to \R^n$ is locally Lipschitz continuous with $f(0) = 0$. The set $\cS$ is assumed to be closed, and convex, so that $\cN_{\cS}(x)$ is defined as
\[
\cN_{\cS}(x) := \{ \eta \in \R^n \, \vert \, \langle \eta, y - x \rangle \le 0, \  \forall y \in \cS \}. 
\]

Since our aim is to provide computational methods for stability analysis, we work with constraint sets $\cS$ which are finitely generated. In particular, the following assumption is imposed throughout the article:
\begin{ass}[leftmargin=*]
\item \label{ass:basicS}The set $\cS$ is convex, contains the origin $\{0\}$, and is described as
\begin{equation}\label{eq:defConSet}
\cS := \{x \in \R^n \, \vert \, g_i(x) \ge 0, i = 1,\dots, M \}
\end{equation}
for some continuously differentiable functions $g_i:\R^n \to \R$. Furthermore, the gradients $\nabla g_i(x) \neq 0$ in some neighborhood of the set $\{x \in \R^n \, \vert \, g_i(x) = 0\}$.
\end{ass}

Based on the discussion following \eqref{eq:compSysGen}, it follows that if $f$ does not admit finite escape time, then there exists a solution to system~\eqref{eq:compSys} which stays in the set $\cS$ for all times. Such a solution corresponds to the particular selection of $\eta$ in \eqref{eq:compSysb}. In other words, if $x$ is on the boundary of the set $\cS$, then the vector $-\eta$ is given by the projection of $f(x)$ onto the normal cone to the set $\cS$, that is, $\cN_{\cS}(x)$, see \citep{BrogTanw20}. For simulation of such systems, an optimization problem is thus solved, at the boundary of the constraint set, to compute $\eta$.

\subsection{Stability Notions}

We first review the stability notions which are to be adapted with respect to the constrained domain, and then provide the definition of Lyapunov functions which we seek for checking stability of system~\eqref{eq:compSys}.

\begin{defn}[Stability]
The origin is stable in the sense of Lyapunov if for every $\eps > 0$ there exists $\delta>0$ such that 
\[
x_0 \in \cS, \|x_{0}\| \le \delta \Rightarrow \|x(t,x_0)\| \le \eps, \forall t \ge t_{0}.
\]
The origin is locally asymptotically stable if it is stable in the sense of Lyapunov and there exists $\rho >0$ such that 
\[
x_0 \in \cS, \|x_{0}\| \le \rho \Rightarrow \lim\limits_{t \rightarrow +\infty} \|x(t,x_0)\| =0.
\]
The origin is globally asymptotically stable if the latter implication holds for arbitrary $\rho > 0$. The origin is globally exponentially stable if there exists $c_0 > 0$ and $\alpha > 0$ such that $\|x(t,x_0)\| \le c_0 e^{-\alpha t} x_0$, for every $x_0 \in \cS$.
\end{defn}

Compared to the conventional definitions of stability for unconstrained dynamical systems, our domain of interest is reduced to the set $\cS$ in system \eqref{eq:compSys}. Also, the vector field jumps instantaneously at the boundaries of the set $\cS$, which may have an impact on the stability of the system. The following example motivates why it is not enough to analyze stability just by looking at the vector field $f$ in \eqref{eq:compSys}, and that the set $\cS$ must also be taken into consideration.

\begin{exam}\label{ex:consGAS}
Let $f(x) = Ax$ with $A = \begin{bsmallmatrix} -1 & -2\\ -1 &-1\end{bsmallmatrix}$, and $\cS = \{x \in \R^2 \, \vert \, 4x_1 - x_2 \ge 0, 4x_2 - x_1\ge 0\}$. Matrix $A$ is not Hurwitz stable since one of its eigenvalues is in the right-half complex plane. However, constrained system \eqref{eq:compSys} is globally asymptotically stable, see our later Example \ref{ex:disret} in Section~\ref{sec:conic} for a proof based on a Lyapunov function. Hence, this example shows that the constraints make the system stable, even if the unconstrained system is unstable.
\end{exam}

Similarly, one can construct examples where the vector field $f$, without any constraints would result in state trajectories converging to the origin, but the presence of constraints makes the dynamical system \eqref{eq:compSys} unstable, for the same vector field. Such examples show that, when developing Lyapunov methods for analyzing stability, it is not enough to just look at the vector field $f$, without looking at set $\cS$.

\subsection{Lyapunov Functions with Constraints}

Based on the above notions, one has to adapt the notion of Lyapunov functions when analyzing the stability of constrained systems of the form~\eqref{eq:compSys}. It is thus of interest to introduce Lyapunov functions which describe the qualitative behavior of the state trajectories on the set $\cS$ only.
With this observation, the following definition of Lyapunov functions for \eqref{eq:compSys} provides more flexibility:

\begin{defn}[Constrained Lyapunov Function]\label{def:funcLyap}
System \eqref{eq:compSys} has a continuously differentiable (global) Lyapunov function $V:\R^n \to \R$ with respect to $\cS$ if
\begin{enumerate}
\item There exist class $\cK_\infty$ functions\footnote{A function $\alpha:\R_+ \to \R_+$ is said to be of class $\cK$ if it is continuous, it satisfies $\alpha(0) = 0$, and it is increasing everywhere on its domain. It is said to be of class $\cK_\infty$ if it is, in addition, unbounded.} $\underline \alpha$, $\overline \alpha$ such that
\[
 \underline \alpha(\| x \| ) \le V(x) \le \overline \alpha(\| x \|), \quad \forall \ x\in \cS;
 \]
\item There exists a class $\cK$ function $\alpha$ such that
\begin{subequations}\label{eq:defLyapIneq}
\begin{gather}
\left\langle \nabla V(x),f(x) \right\rangle \le -\alpha(\| x \|), \ \forall \, x \in \inn(\cS), \\
\left\langle \nabla V(x),f(x)+\eta_x \right\rangle \le -\alpha(\| x \|), \ \forall x \in \bd(\cS),
\end{gather}
where $-\eta_x$ is the projection of $f(x)$ on $\cN_\cS(x)$, such that $f(x) + \eta_x \in \cT_{\cS}(x)$.
\end{subequations}
\end{enumerate}
\end{defn}

In this article, we are interested in computing Lyapunov functions in the sense of Definition~\ref{def:funcLyap} and for that, two classes of algorithms are proposed in Section~\ref{sec:conic} and Section~\ref{sec:sos}, depending upon the structure imposed on the vector field $f$ and the constraint set $\cS$ in \eqref{eq:compSys}.

\section{Conic Sets and Copositive Programming}\label{sec:conic}

We first consider the question of computing the Lyapunov function for system~\eqref{eq:compSys} under the following assumption

\begin{ass}[leftmargin=*]
\item\label{ass:lipf} The function $f:\R^n \to \R^n$ is locally Lipschitz continuous, satisfies $f(0) = 0$, and is homogenous, that is, there exists $d \in \R$ such that for every $\lambda > 0$,
\[
f(\lambda x) = \lambda^d f(x).
\]
\item\label{ass:coneS} The set $\cS$ is a closed convex cone, which we denote by $K$ and is described as
\[
K = \{ x \in \R^n \, \vert \, C x \ge 0\}
\]
for some matrix $C \in \R^{m \times n}$.
\end{ass}

\subsection{Necessary Conditions} 
When addressing the question of computing a Lyapunov function, the first fundamental question is to determine the class of functions where the search should be performed. For the systems of form~\eqref{eq:compSys}, under assumptions \ref{ass:lipf} and \ref{ass:coneS}, an answer to this question appears in our recent work \citep{SouaTanw19pp}. The following statement thus specifies this function class:

\begin{thm}\label{thm:converseHom}
Consider the system~\eqref{eq:compSys} under assumptions \ref{ass:lipf} and \ref{ass:coneS}. If the origin is globally exponentially stable, then there exists a homogenous polynomial $h:\R^n \to \R$, such that, for some non-negative integer $r$, the function
\begin{equation}\label{eq:defVrational}
V(x) = \frac{h(x)}{\|x\|^{2r}}
\end{equation}
is a Lyapunov function for \eqref{eq:compSys}.
\end{thm}

The only information from Theorem~\ref{thm:converseHom}, that we will use in the remainder of this section, is the function class for the Lyapunov function specified in \eqref{eq:defVrational}.

\subsection{Polynomial Inequalities}
To compute the Lyapunov function of the form \eqref{eq:defVrational} numerically, we fix the denominator and reformulate our problem as finding the homogenous polynomial in the numerator which satisfies certain inequalities. We carry out the steps by specifying the inequalities that need to be satisfied, and in the next section, provide the algorithm using convex optimization methods that can be implemented for computing such functions.

Based on the result of Theorem~\ref{thm:converseHom}, we consider $V$ of the form
\[
V(x) = \frac{h(x)}{(\sum_{i=1}^{n} x_i^2)^r} = \frac{h(x)}{\|x\|_{2}^{2r}}
\]
where $h(x)$ is a homogeneous polynomial, and $x=(x_{1},x_{2},\dots,x_{n})^\top \in K$, and $r$ is a non-negative integer. The gradient of this function, denoted by $\nabla V(x) \in \R^n$, is
\[
\nabla V(x) = \frac{\|x\|_{2}^2 \nabla h(x) - 2 r h(x) x}{\|x\|_{2}^{2(r+1)}}.
\]
We now introduce
\[ 
s_0(x)=-\|x\|_{2}^2 \left\langle \nabla h(x), f(x) \right\rangle +2rh(x) \left\langle x, f(x) \right\rangle.
\]
Let $F_i := \{ x \in K \, \vert \, (Cx)_i = 0 \}$, denote a face of the cone $K$, for $i \in \left\{1,\dots,m \right\}$, and let $s_i(x)$ be defined as 
\[
s_i(x)= -\left\langle \|x\|_{2}^2 \nabla h(x) - 2 r h(x) x , f(x) + \eta_x\right\rangle, \quad x \in F_i.
\]
To check the conditions in \eqref{eq:defLyapIneq} of Definition~\ref{def:funcLyap}, and find $V$ of the form \eqref{eq:defVrational}, we thus need to find a homogenous function $h$, and a nonnegative integer $r$, such that
\begin{subequations} \label{Algoineq}
\begin{gather}
 h(x) \ge 0, \enspace \forall x \in K\\
 s_{0}(x) \ge 0, \enspace \forall x \in K \\
 s_{i}(x) \ge 0, \enspace \forall x \in F_{i}, \enspace i \in \left\{1,\dots,m\right\}.
\end{gather}
\end{subequations}

\subsection{Algorithm Description} 

The basic idea behind our algorithm for systems with conic sets, and homogenous vector fields, is to use the structure of the system so that the inequalities in \eqref{Algoineq} need to be checked only for finitely many points over a compact set. In our case, this compact set turns out to be a simplex, or a finite union of simplices.\footnote{An $m$-simplex $\Sigma$ is an $m$-dimensional polytope which is the convex hull of its $m+1$ vertices $\{x^0, x^1, \ldots, x^{m}\}$, namely \[ \Sigma := \left\{\theta_0 x^0 + \dots \theta_m x^{m} \bigg \vert\ \sum_{i=0}^{m} \theta_i =1 \enspace \text{and} \enspace \theta_i \ge 0, \enspace \forall \,  i \in \{1,\ldots,m\}\right\}. \]}
We then select a certain number of points in the simplex and evaluate the inequalities \eqref{Algoineq} with a certain polynomial function parameterized by finitely many unknowns.
This allows us to construct an inner approximation of copositive polynomials with respect to cone $K$.

Because of the conic structure of $K$, we get two nice properties that are desirable for implementing an algorithm:

\begin{itemize}[leftmargin=*]
\item Let $\cO_j$, $j = 1,\dots, 2^n$, denote the orthants of $\R^n$, and let $K_j := \cO_j \cap K$, $F_{ij} := \cO_j \cap F_i$, for $i = 1, \cdots, m$. Then, each $K_j$ and $F_{ij}$ is a closed convex polyhedral cone.
\item For a homogenous polynomial $h \in \R^d[x]$ of degree $d$, it holds that
\begin{equation}\label{eq:norm1}
h (x) \ge 0, \ \forall \, x \in K \iff h(x) \ge 0, \ \forall \,  x \in K, \ \|x\|=1.
\end{equation}
\end{itemize}
As a result of these properties, it is convenient to introduce the simplices obtained by intersecting the cones $K_j$ or $F_{ij}$ with the set $\{x \in\R^n \, \vert \, \| x \|_{1} =1\}$, that is,
\[
\Sigma_j := \left\{x \in K_j \vert\ \|x\|_{1}=1\right\}, \quad \Sigma_{ij} := \left\{x \in F_{ij} \vert\ \|x\|_{1}=1\right\}.
\]
We next reduce the task to checking the inequalities on a finite number of points in each of the simplex $\Sigma_j$ and $\Sigma_{ij}$.

Because of equivalence in \eqref{eq:norm1}, positivity of a homogenous polynomial $h$ is then expressed as
\[
h(x) \ge 0 \enspace \text{for all} \enspace x \in \bigcup_{j=1}^{2^n} \Sigma_j \cup \big(\cup_{i=1}^m \Sigma_{ij}\big)
\]

\subsubsection{Simplex Discretization:}
Our goal is to discretize the simplex $\Sigma$ and obtain a hierarchy of linear inequalities with respect to the discretization points which allow us to find the desired function.

\begin{defn} \label{simplicialpartition}
Let $\Sigma$ be a simplex in $\R^n$. A family $\cP=\left\{\Delta^1,\dots,\Delta^m\right\}$ of simplices satisfying
\[
\Sigma=\bigcup_{i=1}^m \Delta^i \enspace \text{and} \enspace \inn\Delta^i \cap \inn\Delta^j=\emptyset \enspace \text{for} \enspace i \ne j
\]
is called a simplicial partition of $\Sigma$.
\end{defn}

For a simplicial partition $\cP=\{\Delta^1,\dots,\Delta^m\}$ of $\Sigma$, we let $W_{\cP}$ denote the set of all vertices of simplices in $\cP$, and $E_{\cP}$ the set of all edges of simplices in $\cP$. The cardinality of $W_{\cP}$ is $p=|W_{\cP}|$.

\subsubsection{Tensor representation:} As a final tool, we introduce {\em tensors}, which generalize the notion of a matrix, and will be used for compact representation of polynomials and its values on the vertices of simplical partition.

\begin{defn}
A tensor $H$ of order $d$ over $\R^n$ is a multilinear form
\begin{equation*}
\begin{array}{ccccc}
& \underbrace{\R^n \times \R^n \times \dots \times \R^n}_{d \text{ times}} & \to & \R \\
 & (x^1,x^2,\dots,x^d) & \mapsto & H[x^1,x^2,\dots,x^d] \\
\end{array}
\end{equation*}
where 
\[
H[x^1,x^2,\dots,x^d]= \sum_{i_1=1}^n \sum_{i_2=1}^n \sum_{i_d=1}^n h_{i_1,i_2,\dots,i_d} x^1_{i_1} \cdots x^d_{i_d}
\]
and $h_{i_1,i_2,\dots,i_d}$ corresponds to a real number from a table with $n^d$ entries, indexed by $i_1,i_2,\dots,i_d \in \{1, \dots, n\}$. We say that $H$ is symmetric if
\[
h_{i_1,i_2,\dots,i_d} = h_{j_1,j_2,\dots , j_d}
\]
whenever $i_1+i_2+\dots+i_d = j_1+j_2+\dots+j_d$, for all possible permutations $i_1,i_2,\dots,i_d$ and $j_1,j_2,\dots , j_d$ of $\{1, \dots, n\}$.
\end{defn}
A matrix $H \in \R^{n \times n}$ describes a tensor of order $2$ over $\R^n$, also called a quadratic form, where the coefficients of the quadratic form belong to a table with $n^2$ entries $a_{i,j}$ with $i,j=\{1,\dots, n\}$. 
A general homogeneous polynomial $h \in \R^d[x]$, with $d \ge 2$, can be written as
\[
h(x)=h(x_1,\dots,x_n)=\sum \limits_{\underset{i_{1}+\dots+i_{n}=d}{i=(i_1, \dots, i_n)}} h_{i} x_{1}^{i_1} \cdots x_{n}^{i_n}.
\]
Using the tensor representation, $h$ can also be compactly written in the form
\begin{equation}\label{tensorpol}
h(x)=H[\underbrace{x,x,\dots,x}_{\text{$d$ times}}]
\end{equation}
where $H$ is a symmetric tensor.

\subsubsection{Constructing the Hierarchy:} For a fixed partition $\cP = \{\Delta^1, \cdots, \Delta^\ell\}$ of the simplex $\Sigma$, we look at the vertices of $\Delta^k$, and evaluate the tensor at all possible combinations of the vertices within $\Delta^k$. The following proposition provides a feasibility check for this method:
\begin{lem} \label{simpconv}
For a simplicial partition $\cP$, with set of vertices $W_{\cP} = \{ v_1, \cdots, v_p\}$, and $\Delta=\mathrm{conv}\{v_{1},\dots,v_{p}\}$. If
\begin{equation} \label{assum}
H[v_{i_1},v_{i_2},\dots,v_{i_d}] \ge 0 \enspace \text{for all} \enspace i_1,i_2,\dots,i_d=1,\dots,p,
\end{equation}
then $h(x)= H[x,x,\dots,x] \ge 0$ for all $x \in \Delta$.
\end{lem}

\begin{proof}
For each point $x \in \Delta$, we can represent it in the affine hull of $\Delta$ by its uniquely determined barycentric coordinates $\lambda=(\lambda_{1},\dots,\lambda_{p})$ with respect to $\Delta$ i.e.
\[
x=\sum_{j=1}^{p} \lambda_{j}v_{j} \enspace \text{with} \enspace \sum_{j=1}^{p} \lambda_{j}=1.
\]
This gives
\begin{align*}
h(x)
&=H[x,x,\dots,x]\\
&=H\big[\sum_{i_1=1}^{p} \lambda_{i_1}v_{i_1},\sum_{i_2=1}^{p} \lambda_{i_2}v_{i_2},\dots,\sum_{i_d=1}^{p} \lambda_{i_d}v_{i_d}\big]\\
&=\sum_{i_1,i_2,\dots,i_d=1}^{p}H[v_{i_1},v_{i_2},\dots,v_{i_d}]\lambda_{i_1}\lambda_{i_2} \dots \lambda_{i_d}.
\end{align*}
For $x \in \Delta$, we have $\lambda_i \ge 0$, and by the assumption \eqref{assum}, we get $h(x) \ge 0$ for all $x \in \Delta$.  
\end{proof}

The observation of Lemma~\ref{simpconv} leads to the following algorithm for computing the copositive Lyapunov function of the form \eqref{eq:defVrational} satisfying the inequalities \eqref{Algoineq}.

{\bf Algorithm~1:} 
\begin{enumerate}
\item Take $h \in \R[x]$, homogenous of degree $d$, and fix $r \in \N$.
\item For each orthant $\cO_j$, $j =1,\dots,2^n$, compute the sets $K_j = K \cap \cO_j$ and for each $i = 1,\dots, m$, let $F_{ij} = F_i \cap \cO_j$.
\item Identify the simplices $\Sigma_j \subset K_j$, and $\Sigma_{ij} \subset F_{ij}$ which are non-empty.
\item For each nonempty simplex $\Sigma \in \{ \Sigma_j\} \cup \{ \Sigma_{ij}\}$, $j = 1,\dots, 2^n$, $i = 1,\dots,m$,
\begin{enumerate}
\item Compute a simplical partitioning of the set $\Sigma$, denoted by $\{\Delta^1,\dots, \Delta^{\ol{\ell}}\}$, and let $Q^\ell$ be the corresponding set of vertices of $\Delta^\ell$.
\item For each set of $d$ vertices $\{q_1, \dots, q_d\} \in Q^\ell$, solve the LP problem in the coefficients of $h$ corresponding to the constraints 
\begin{equation}\label{eq:checkPosDisc}
H[q_1,\dots,q_d] \ge 0, \quad \text{and } S_k[q_1,\dots,q_d] \ge 0
\end{equation}
where $H,S_k$ denote the tensors of $h,s_k$, $k = 0, \dots, m$ and $\left\{q_1,\dots,q_d\right\} \in Q^\ell$. 
\item If \eqref{eq:checkPosDisc} is infeasible, refine partition, and check \eqref{eq:checkPosDisc} again.
\end{enumerate}
\item Iterate by increasing $d$ and $r$.
\end{enumerate}

As an illustration of our algorithm, we revisit Example~\ref{ex:consGAS} and compute a quadratic Lyapunov function using the discretization method.
\begin{exam}\label{ex:disret}
Consider system \eqref{eq:compSys} with $f(x)=Ax$ and $A = \begin{bsmallmatrix} -1 & -2\\ -1 &-1\end{bsmallmatrix}$ and $K=\{x \in \R^2 \, \vert \, Cx \ge 0 \}$, with $C = \begin{bsmallmatrix} -0.25 & 1\\ 1 & -0.25 \end{bsmallmatrix}$. We apply the discretization method on the three simplices that correspond to $K_j = K \cap \cO_j$,
\begin{gather*}
\Sigma_1= \mathrm{conv}([1,0]^\top, [0,1]^\top),  \Sigma_2 = \mathrm{conv}([1,0]^\top, [0.8,-0.2]^\top)\\
\Sigma_3= \mathrm{conv}([0,1]^\top, [-0.2,0.8]^\top),
\end{gather*}
and the two simplices which correspond to the two faces of the cone reduce to a singleton, that is,
\[
\Sigma_{12} = {[0.8,-0.2]}, \quad \Sigma_{23} = [-0.2,0.8].
\]
Solving the resulting inequalities, we obtain
\begin{equation}
V(x) = 2.9 x_{1}^2 + x_{1}x_{2}+x_{2}^2,
\end{equation}
which indeed satisfies the inequalities in \eqref{Algoineq}.
\end{exam}

\section{Semialgebraic Sets and Sum-of-Squares Computation}\label{sec:sos}

We now present a numerical approach to deal with sets of the form $\cS$ in \eqref{eq:defConSet} where $g_i$ are not necessarily linear. One of the difficulties in checking the Lyapunov conditions is that the corresponding inequality has to be checked for $\eta \in -\cN_{\cS}(x)$ for all $x \in \cS$, which is not feasible in general. In the previous section, we only need to check the inequalities at finitely many points at which $\eta$ could be obtained as a solution to an optimization problem, but this works only under the conic structure of $\cS$. For more general sets without conic structure, it is of interest to obtain Lyapunov functions without having to solve for $\eta$. One way to avoid computation of $\eta$ is to impose certain assumption on the gradient of Lyapunov function and provide sufficient conditions which can be checked independently of $\eta$. We then use these conditions to compute Lyapunov functions using a semidefinite program based on SOS decomposition.

\subsection{Sufficient Conditions}

With the aforementioned motivation, we first provide a set of inequalities as a sufficient condition for checking  asymptotic stability of \eqref{eq:compSys}, which are independent of $\eta$ and use the information of the gradients of the generating functions $g_i$, $i = 1,\dots, M$.

\begin{prop}[Sufficient Conditions]\label{prop:suffGrad}
Consider the system~\eqref{eq:compSys} under assumption \ref{ass:basicS}. Assume that there exists a continuously differentiable $V(\cdot)$ that satisfies the following conditions:
\begin{itemize}
\item $V(0) = 0$, and $ \ul \alpha (\|x \|) \le V(x) \le \ol \alpha(\| x \|)$ for every $x \in \cS$, and some class $\cK$ functions $\ul\alpha$, $\ol \alpha$.
\item $\langle f(x), \nabla V(x) \rangle \le - \alpha \| x \| $, for every $x \in \cS$, and some positive definite function $\alpha$.
\item If $x$ is such that $g_i(x) = 0$, for some $i \in \{1,\cdots, M\}$, then  $\langle \nabla g_i(x), \nabla V(x) \rangle  \le 0$.
\end{itemize}
Then $V$ is a Lyapunov function for system \eqref{eq:compSys} and $0$ is globally asymptotically stable.
\end{prop}

\begin{proof}
Consider a function $V \in \cC^1(\R^n, \R)$ that satisfies the listed conditions. We show that these conditions guarantee that $V$ is Lyapunov function for system \eqref{eq:compSys} when $\cS$ is described by \eqref{eq:defConSet} under assumption \ref{ass:basicS}. To see this, we first introduce the set $J(x)$ which defines the set of active constraints, that is,
\begin{equation}\label{eq:defActiveSet}
J(x) = \{i \in \{1,\dots,M\} \, \vert \, g_i(x) = 0\}.
\end{equation}
Then, the set-valued mapping $\cN_\cS$ is defined as
\[
\cN_\cS(x) = \begin{cases} 0, & \text{if } x \in \inn(\cS),\\ 
\Big\{\sum_{j \in J(x)} \! \! \lambda_j \nabla g_j(x); \enspace \lambda_j \le 0\Big\}, & \text{if } J(x) \neq \emptyset, \\
\emptyset, & \text{if } x \not\in \cS.
\end{cases}
\]
Thus, if $x \in \inn(S)$, then $\eta = 0$, and 
\[
\langle \nabla V(x), f(x) \rangle \le -\alpha(\| x \|), \quad x \in \inn(S).
\]
When $x$ is such that $J(x) \neq \emptyset$, we have that $\eta = -\sum_{j \in J(x)} \lambda_j \nabla g_j(x)$, for some $\lambda_j \le 0$. Hence,
\[
\langle \eta, \nabla V(x) \rangle = - \left\langle \sum_{j \in J(x)} \lambda_j \nabla g_j(x), \nabla V(x) \right\rangle \le 0.
\] 
Thus, for each $x \in \cS$, and $\eta \in -\cN_{\cS}(x)$, we have shown that
\[
\langle \nabla V(x), f(x) + \eta \rangle \le - \alpha(\| x \|)
\]
which completes the proof.  
\end{proof}

\subsection{Sum-of-Squares Decomposition}

We now present a numerical approach to compute the Lyapunov function which satisfies the conditions of  Proposition~\ref{prop:suffGrad}. The three conditions can be listed as positivity constraints on the function $V$ and its gradient $\nabla V$. As mentioned in the introduction, one way to ensure the positivity is to write the function as a sum-of-squares, which boils down to a semidefinite program. The basic idea behind computing the Lyapunov function for system~\eqref{eq:compSys} under \ref{ass:basicS} is to find a Lyapunov function where the three positivity constraints in Proposition~\ref{prop:suffGrad} can be written as sum-of-squares.

We focus our attention on convex semi-algebraic sets, which are basically described by the intersection of the sublevel sets of finitely many polynomial inequalities. That is, in the definition of the set $\cS$ in \eqref{eq:defConSet}, we introduce the following assumption:
\begin{ass}[leftmargin=*]
\item\label{ass:semiAlg} The set $\cS$ in \eqref{eq:defConSet} is compact and the function $g_i \in \R[x]$, for every $i = 1, \dots, M$.
\end{ass}
For such sets, we can implement the following algorithm to compute $V$ in the form of sum-of-squares.\\
{\bf Algorithm 2:} 
\begin{enumerate}
\item Let $V\in \R[x]$ of degree $d \in \N$;
\item For each $x \in \cS$, let 
\[
V(x) = \sigma_0(x) + \sum_{i=1}^M \sigma_i(x) g_i(x).
\]
for some SOS polynomials $\sigma_0, \cdots, \sigma_M$.
\item For each $x \in \cS$, if $J(x) = \emptyset$, let 
\[
 - \langle \nabla V(x),f(x) \rangle = \chi_0(x) + \sum_{i=1}^M \chi_i(x) g_i(x).
\]
for some SOS polynomials $\chi_0, \cdots, \chi_M$.
\item For each $x \in \cS$, $J(x) \neq \emptyset$, let, for each $j \in J(x)$,
\begin{equation}
- \langle \nabla V(x),\nabla g_j(x) \rangle = \chi_{j,0} (x)+ \sum_{i \not\in J(x)} \chi_{j,i} (x) g_i(x) + \sum_{i \in J(x)}\varphi_{j,i} g_i(x),
\end{equation}
for some SOS polynomials $\chi_{j,i}$, whereas $\varphi_{j,i} \in \R[x]$ are not necessarily sum-of-squares.
\item Iterate by increasing $d$, the degree of $V$.
\end{enumerate}

An important question to consider, in the implementation of Algorithm~2, is whether one can always find SOS decomposition of a positive polynomial on a semialgebraic set. One possible answer to this question comes from the following result: 

\begin{thm}\label{thm:PutinarPsatz}(Putinar's Positivstellensatz~\citep{Puti93})
Let $\cS$ be a compact semialgebraic set satisfying \ref{ass:basicS} and \ref{ass:semiAlg}. Let $\cM_\cS$ be the quadratic module defined as,
\[
\cM_\cS := \Big\{\sigma_0 + \sum_{j=1}^m \sigma_j g_j \, \vert \, j = 0, 1, \dots, m \Big\}.
\]
Let $V \in \R[x]$ be such that $V(x) \ge 0$ for all  $x \in \cS$, then $V \in \cM_\cS$.
\end{thm}
A direct application of this result to our problem suggests that, if system~\eqref{eq:compSys} admits a polynomial Lyapunov function, then the hierarchy of semidefinite programs constructed in Algortihm~2 (by increasing the degree $d$ of the search function) is guaranteed to find us a Lyapunov function. To compute $V$ with such a parameterization, one may use the YALMIP toolbox in Matlab to solve the underlying semidefinite program.

\begin{exam}
As an illustration of the foregoing algorithm, we consider an academic example in $\R^2$ with two constraints. Let $g_1(x) = x_1-x_2^2$, and $g_2(x) = 1-x_1$. These two functions describe the compact semi-algebraic set $\cS$ in \eqref{eq:compSys}, and we take vector field $f$ to be
$
f(x) = \begin{psmallmatrix} -x_1^2 \\ 0 \end{psmallmatrix}.
$
Based on Algorithm~2, a Lyapunov function for this example is
$
V(x) = x_1^2 + x_2^2
$
which indeed satisfies the conditions listed in Proposition~\ref{prop:suffGrad}. Note that the system without constraints, that is, $\dot x = f(x)$ is only stable, but not asymptotically stable. However, the constrained system is asymptotically stable since, within the set $\cS$, $x_1 = 0$ implies $x_2 = 0$.
\end{exam}



\Urlmuskip=0mu plus 1mu\relax


{\small

\begin{thebibliography}{23}
\providecommand{\natexlab}[1]{#1}
\providecommand{\url}[1]{\texttt{#1}}
\providecommand{\urlprefix}{URL }
\expandafter\ifx\csname urlstyle\endcsname\relax
  \providecommand{\doi}[1]{doi:\discretionary{}{}{}#1}\else
  \providecommand{\doi}{doi:\discretionary{}{}{}\begingroup
  \urlstyle{rm}\Url}\fi

\bibitem[{Acary et~al.(2011)Acary, Bonnefon, and Brogliato}]{acary2011}
Acary, V., Bonnefon, O., and Brogliato, B. (2011).
\newblock \emph{Nonsmooth Modeling and Simulation for Switched Circuits},
  volume~69 of \emph{Lecture Notes in Electrical Engineering}.
\newblock Springer, Heidelberg.

\bibitem[{Adly(2017)}]{adly2017}
Adly, S. (2017).
\newblock \emph{A Variational Approach to Nonsmooth Dynamics. Applications in
  Unilateral Mechanics and Electronics}.
\newblock Springer Briefs in Mathematics. Springer International Publishing,
  Cham.

\bibitem[{Ahmadi and Jungers(2018)}]{AhmaJung18}
Ahmadi, A. and Jungers, R. (2018).
\newblock {SOS}-{C}onvex {L}yapunov functions and stability of difference
  inclusions.
\newblock Available online: \url{https://arxiv.org/abs/1803.02070}.

\bibitem[{Ahmadi and Parrilo(2017)}]{AhmPar17}
Ahmadi, A. and Parrilo, P. (2017).
\newblock Sum of squares certificates for stability of planar, homogeneous, and
  switched systems.
\newblock \emph{IEEE Transactions on Automatic Control}, 62(10), 5269--5274.

\bibitem[{Brogliato and Tanwani(2020)}]{BrogTanw20}
Brogliato, B. and Tanwani, A. (2020).
\newblock Dynamical systems coupled with monotone set-valued operators:
  Formalisms, applications, well-posedness, and stability.
\newblock \emph{SIAM Review}, 62, 3--129.

\bibitem[{Bundfuss and D\"{u}r(2008)}]{BunDur08}
Bundfuss, S. and D\"{u}r, M. (2008).
\newblock Algorithmic copositivity detection by simplicial partition.
\newblock \emph{Linear Algebra and its Applications}, 428, 1511--1523.

\bibitem[{Bundfuss and D\"{u}r(2009)}]{BunDur09}
Bundfuss, S. and D\"{u}r, M. (2009).
\newblock An adaptive linear approximation algorithm for copositive programs.
\newblock \emph{SIAM J. Optim.}, 20(1), 30--53.

\bibitem[{Camlibel et~al.(2006)Camlibel, Pang, and Shen}]{CamlPang06}
Camlibel, M., Pang, J.S., and Shen, J. (2006).
\newblock Lyapunov stability of complementarity and extended systems.
\newblock \emph{{SIAM} J.~Optimization}, 17(4), 1056--1101.

\bibitem[{Chesi et~al.(2009)Chesi, Garulli, Tesi, and Vicino}]{Chesi09}
Chesi, G., Garulli, A., Tesi, A., and Vicino, A. (2009).
\newblock \emph{Homogeneous Polynomial Forms for Robustness Analysis of
  Uncertain Systems}.
\newblock LNCIS. Springer.

\bibitem[{Goeleven and Brogliato(2004)}]{GoelBrog04}
Goeleven, D. and Brogliato, B. (2004).
\newblock Stability and instability matrices for linear evolution variational
  inequalities.
\newblock \emph{{IEEE} Transactions on Automatic Control}, 49(4), 521--534.

\bibitem[{Goeleven et~al.(2003)Goeleven, Motreanu, and Motreanu}]{GoelMotr03}
Goeleven, D., Motreanu, D., and Motreanu, V. (2003).
\newblock On the stability of stationary solutions of first-order evolution
  variational inequalities.
\newblock \emph{Advances in Nonlin. Variational Ineqs.}, 6, 1--30.

\bibitem[{Henrion(2013)}]{Henr13}
Henrion, D. (2013).
\newblock Optimization on linear matrix inequalities for polynomial systems
  control.
\newblock \emph{Les cours du {C.I.R.M.}}, 3(1), 1--44.

\bibitem[{Henrion and Garulli(2005)}]{HenrGaru05}
Henrion, D. and Garulli, A. (eds.) (2005).
\newblock \emph{Positive Polynomials in Control}.
\newblock Springer, New York, USA.

\bibitem[{Lasserre(2015)}]{Lass15}
Lasserre, J. (2015).
\newblock \emph{An Introduction to Polynomial and Semi-Algebraic Optimization}.
\newblock Cambridge Texts in Applied Maths.

\bibitem[{Leine and van~de Wouw(2008)}]{leine2008}
Leine, R.I. and van~de Wouw, N. (2008).
\newblock \emph{Stability and Convergence of Mechanical Systems with Unilateral
  Constraints}, volume~36 of \emph{Lecture Notes in Applied and Computational
  Mechanics}.
\newblock Springer-Verlag, Berlin Heidelberg.

\bibitem[{Nie et~al.(2018)Nie, Yang, and Zhang}]{NiYaZha18}
Nie, J., Yang, Z., and Zhang, X. (2018).
\newblock A complete semidefinite algorithm for detecting copositive matrices
  and tensors.
\newblock \emph{SIAM J. Optim.}, 28(4), 2902--2921.

\bibitem[{Papachristodoulou and Prajna(2009)}]{PapaPraj09}
Papachristodoulou, A. and Prajna, S. (2009).
\newblock Robust stability analysis of nonlinear hybrid systems.
\newblock \emph{{IEEE} Transactions on Automatic Control}, 54(5), 1037--1043.

\bibitem[{Parrilo(May 2000)}]{Parr00}
Parrilo, P. (May 2000).
\newblock \emph{Structured semidefinite programs and semi-algebraic geometry
  methods in robustness and optimization}.
\newblock Ph.D. thesis, California Institute of Technology, Pasadena, CA.

\bibitem[{Powers and W\"ormann(1998)}]{Powers98}
Powers, V. and W\"ormann, T. (1998).
\newblock An algorithm for sums of squares of real polynomials.
\newblock \emph{J. of Pure and Applied Algebra}, 127(1), 99--104.

\bibitem[{Prajna et~al.(2002)Prajna, Papachristodoulou, and
  Parrilo}]{PrajPapa02}
Prajna, S., Papachristodoulou, A., and Parrilo, P.A. (2002).
\newblock {SOSTOOLS}: {S}um of squares optimization toolbox for {MATLAB}.
\newblock Available online: \url{http://www.cds.caltech.edu/sostools}.

\bibitem[{Putinar(1993)}]{Puti93}
Putinar, M. (1993).
\newblock Positive polynomials on compact semi-algebraic sets.
\newblock \emph{Indiana Univ.~Math.~Journal}, 42(3), 969--984.

\bibitem[{Souaiby et~al.(2019)Souaiby, Tanwani, and Henrion}]{SouaTanw19pp}
Souaiby, M., Tanwani, A., and Henrion, D. (2019).
\newblock Cone-copositive Lyapunov functions for complementarity systems:
  {C}onverse result and Polynomial Approximation.
\newblock Submitted for publication. Available online:~
\url{https://hal.archives-ouvertes.fr/hal-02565283}

\bibitem[{Tanwani et~al.(2018)Tanwani, Brogliato, and Prieur}]{TanwBrog18}
Tanwani, A., Brogliato, B., and Prieur, C. (2018).
\newblock Well-posedness and output regulation for implicit time-varying
  evolution variational inequalities.
\newblock \emph{SIAM J. Control \& Optim.}, 56(2), 751--781.

\end{thebibliography}



}
\end{document}